\documentclass[12pt,reqno]{amsart}

\usepackage{color}
\definecolor{vert}{rgb}{0,0.6,0}

\usepackage{tikz}

\usepackage{comment}
\usepackage{amsmath}
\usepackage{amssymb}
\usepackage{amsthm}
\usepackage[latin1]{inputenc}
\usepackage{eurosym}
\usepackage{graphicx}
\usepackage{bmpsize}
\usepackage{epsfig}
\usepackage{hyperref}
\usepackage{dsfont}
\usepackage[nocompress, space]{cite}
\usepackage{caption}
\usepackage{subcaption}
\usepackage{cases}
\usepackage{bm} 

\usepackage[displaymath,mathlines]{lineno}
\allowdisplaybreaks

\usepackage{ifthen}

\newcommand{\N}{\mathbb{N}}

\newcommand{\R}{\mathbb{R}}

\newcommand{\Z}{\mathbb{Z}}

\newcommand{\cA}{\mathcal{A}}

\newcommand{\cR}{\mathcal{R}}

\newcommand{\gam}{\gamma}

\newcommand{\ep}{\varepsilon}

\newcommand{\Lip}{{\rm Lip\,}}

\newcommand{\ol}{\overline}

\newcommand{\argmax}{\operatorname{argmax}}
\newcommand{\Div}{{\rm div}\,}

\renewcommand{\div}{\operatorname{div}}

\def\leq{\leqslant}
\def\geq{\geqslant}

\numberwithin{equation}{section}

\newtheoremstyle{thmlemcorr}{10pt}{10pt}{\itshape}{}{\bfseries}{.}{10pt}{{\thmname{#1}\thmnumber{
#2}\thmnote{ (#3)}}}
\newtheoremstyle{thmlemcorr*}{10pt}{10pt}{\itshape}{}{\bfseries}{.}\newline{{\thmname{#1}\thmnumber{
#2}\thmnote{ (#3)}}}
\newtheoremstyle{defi}{10pt}{10pt}{\itshape}{}{\bfseries}{.}{10pt}{{\thmname{#1}\thmnumber{
#2}\thmnote{ (#3)}}}
\newtheoremstyle{remexample}{10pt}{10pt}{}{}{\bfseries}{.}{10pt}{{\thmname{#1}\thmnumber{
#2}\thmnote{ (#3)}}}
\newtheoremstyle{ass}{10pt}{10pt}{}{}{\bfseries}{.}{10pt}{{\thmname{#1}\thmnumber{
A#2}\thmnote{ (#3)}}}

\theoremstyle{thmlemcorr}
\newtheorem{theorem}{Theorem}
\numberwithin{theorem}{section}
\newtheorem{lemma}[theorem]{Lemma}

\newtheorem{proposition}[theorem]{Proposition}

\theoremstyle{thmlemcorr*}
\newtheorem{theorem*}{Theorem}
\newtheorem{lemma*}[theorem]{Lemma}
\newtheorem{corollary*}[theorem]{Corollary}
\newtheorem{proposition*}[theorem]{Proposition}
\newtheorem{problem*}[theorem]{Problem}
\newtheorem{conjecture*}[theorem]{Conjecture}

\theoremstyle{defi}
\newtheorem{definition}[theorem]{Definition}

\theoremstyle{remexample}
\newtheorem{remark}{Remark}
\newtheorem{example}[theorem]{Example}

\theoremstyle{plain}
\newtheorem{thm}[theorem]{Theorem}

\theoremstyle{ass}

\begin{document}

\title[Forced mean curvature flow with evolving spirals]{Asymptotic growth rate of solutions to level-set forced mean curvature flows with evolving spirals}

\author{Hiroyoshi Mitake}
\address[H. Mitake]{
	Graduate School of Mathematical Sciences, 
	University of Tokyo 
	3-8-1 Komaba, Meguro-ku, Tokyo, 153-8914, Japan}
\email{mitake@g.ecc.u-tokyo.ac.jp}

\author{Hung V. Tran}
\address[Hung V. Tran]
{
	Department of Mathematics, 
	University of Wisconsin Madison, Van Vleck hall, 480 Lincoln drive, Madison, WI 53706, USA}
\email{hung@math.wisc.edu}

\thanks{
	The work of HM was partially supported by the JSPS grants: KAKENHI \#22K03382, \#21H04431, \#20H01816, \#19H00639.
	The work of HT was partially supported by  NSF CAREER grant DMS-1843320 and a Vilas Faculty Early-Career Investigator Award.
}

\date{\today}

\begin{abstract}
Here, we study a level-set forced mean curvature flow with evolving spirals and the homogeneous Neumann boundary condition, which appears in a crystal growth model.
Under some appropriate conditions on the forcing term, we prove that the solution is globally Lipschitz.
We then study the large time average of the solution and deduce the asymptotic growth rate of the crystal.
Some large time behavior results of the solution are obtained.

\end{abstract}

\maketitle

%%%%%%%%%%%%%%%%%%%%%%%%%%%%%%%%%%%%%%%%%%%%%%%%%%%%%%%%%%%%%%%%%%%%%

\section{Introduction} \label{sec:intro}
\subsection{A spiral crystal growth model}
Various spiral patterns are observed in many crystal growth situations in practice in which the spirals' centers are often believed to be the locations of screw dislocations.
In this growth mechanism, the crystal surface has discontinuities in height, which are called steps, along curves that spiral out from these centers. 
Atoms bond with the crystal structure with a higher probability near each step, which results in a surface evolution of this step with normal velocity
\begin{equation}\label{surface-ev}
V=v_\infty(\rho_c \kappa+1),  
\end{equation}
where $\kappa$ is the curvature of the step.
 The constants $\rho_c>0$ and $v_\infty>0$ are the critical radius and the step velocity \cite{BCF}, respectively. 
 By rescaling, we may assume that $v_\infty\rho_c=1$, and thus, \eqref{surface-ev} becomes $V = \kappa + v_\infty$.
 See Figure \ref{fig.spiral} for an example of a spiral crystal growth with only one screw dislocation.

 \begin{center} 
\includegraphics[width=6in]{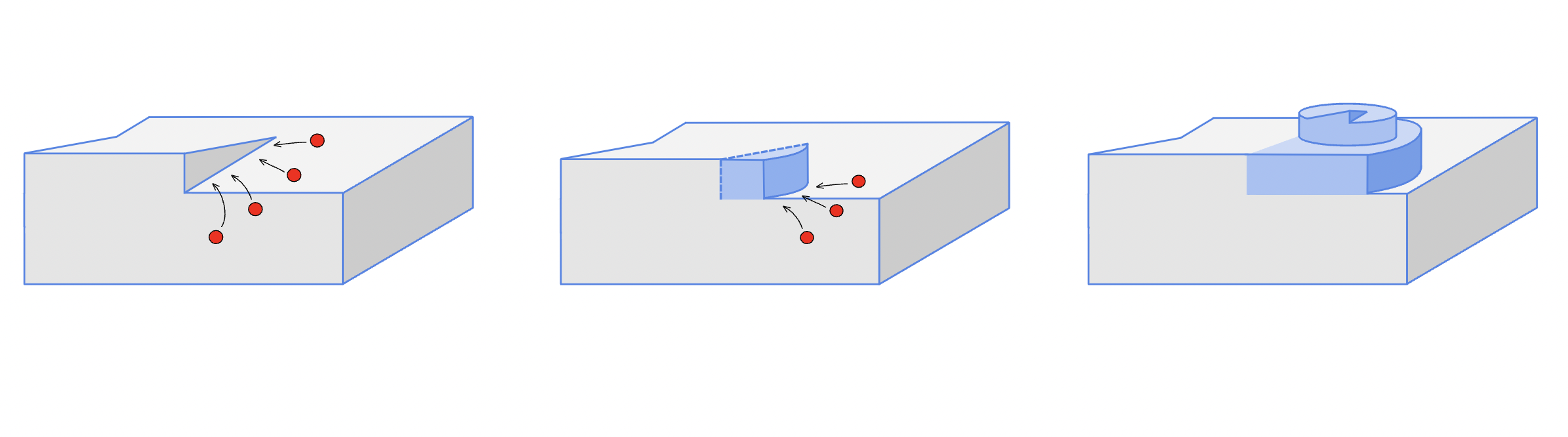}
\captionof{figure}{Spiral crystal growth with a screw dislocation.}
\label{fig.spiral}
\end{center}

We now give a more precise description of the spirals and the set up of the problem following \cite{Oh, OTG1, OTG2}.
Let $\Omega \subset \R^2$ be a bounded domain with smooth boundary.
Let $N\in \N$ and $a_1,\ldots, a_N \in \Omega$ be the centers of the spirals with $a_i \neq a_j$ for $i\neq j$.
Fix $r>0$ sufficiently small so that $\ol B(a_j,r) \subset \Omega$ for $1\leq j \leq N$, 
and $\ol B(a_i,r) \cap \ol B(a_j,r) =\emptyset$ for $i \neq j$.
Let
\begin{align*}
& 
W:=\Omega\setminus \bigcup_{j=1}^N B(a_j,r) \subset\R^2, \\ 
&
\theta(x):= 
\sum_{j=1}^N m_j\arg(x-a_j),
\end{align*}
where $m_j\in\Z \setminus\{0\}$ is given for $1\leq j \leq N$.
The constant $m_j=m_j^+-m_j^-$  with $m_j^\pm \in \N\cup\{0\}$ is the strength of the spiral center $a_j$, which is the difference between the strength $m_j^+$ of counterclockwise rotating spirals and $m_j^-$ of clockwise rotating spirals.
Note that $\theta$ is a multi-valued function and it is not well defined at $a_1,\ldots, a_N$, and this is the reason why we consider the problem in $W$ instead of $\Omega$.
Although $\theta$ is multi-valued, $D\theta$ is single-valued. 
Indeed, if $f(x)=\arg(x)$, then we can view that
$f(x)=\arctan(x_2/x_1)+2\pi l$ for some $l\in\Z$. 
Thus, 
\[
f_{x_1}=\frac{1}{1+\left(\frac{x_2}{x_1}\right)^2}\cdot\left(-\frac{x_2}{(x_1)^2}\right)
=\frac{-x_2}{|x|^2}, \quad
f_{x_2}=\frac{1}{1+\left(\frac{x_2}{x_1}\right)^2}\cdot\left(\frac{1}{x_1}\right)
=\frac{x_1}{|x|^2}. 
\]
Hence, for $a_j=(a_{j,1}, a_{j,2})$ for $1\leq j \leq N$.
\[
D\theta(x)=
\sum_{j=1}^N
\frac{m_j (-(x_2-a_{j,1}), x_1-a_{j,2})}{|x-a_j|^2}
=
\sum_{j=1}^N
\frac{m_j (x-a_j)^\perp}{|x-a_j|^2}.
\]
As noted, $D\theta$ is well defined and smooth in $W$. 
For an auxiliary function $u:W \times [0,\infty) \to \R$, spirals are implicitly defined as, for $t\geq 0$,
\begin{equation}\label{eq:Gamma_t}
\Gamma_t =\{x\in \ol W\,:\, u(x,t)-\theta(x) =2\pi k \text{ for some } k\in \Z\}.
\end{equation}

 In this paper, we consider a more general setting in which the step velocity might be non homogeneous, that is, $v_\infty = c(x)$ for $x\in W$.
 Then, $V=\kappa + c(x)$, and the auxiliary function $u$ satisfied the following level-set forced mean curvature flow PDE
\begin{equation}\label{eq:C}
\begin{cases}
u_t =|D(u-\theta)| \left(\Div\left(\frac{D(u-\theta)}{|D(u-\theta)|}\right)+c(x)\right) \qquad &\text{ in } W\times(0,\infty), \\
D(u-\theta)\cdot \mathbf{n}=0  \qquad &\text{ on } \partial W\times(0,\infty),\\
u(\cdot,0)=u_0 \qquad &\text{ on }  \ol W,
\end{cases}
\end{equation}
Here, ${\mathbf{n}}$ is the outward unit normal vector to $\partial W$. 
Throughout this paper, we assume that $c \in C^1(\ol W, (0,\infty))$ and $u_0\in C^{2}(\overline{W})$ with $D(u_0-\theta)\cdot \mathbf{n}=0$ on $\partial W$   for compatibility. 
The wellposedness of viscosity solutions to \eqref{eq:C} was obtained in \cite{Oh}.
The uniqueness of level sets $\{\Gamma_t\}_{t>0}$ with respect to an initial curve $\Gamma_0$ was proved in \cite{GNOh}.

Our main goals are to study the Lipschitz regularity, the large time average and behavior of the auxiliary function $u$, and to deduce the asymptotic growth rate of the crystal.

\subsection{Main results}
We now describe our main results. 
Denote by
\begin{equation*}
\begin{cases}
C_0:=\max\{-\lambda\,:\,\lambda\ \textrm{is the curvature of $\partial W$ at $x_0$ for $x_0 \in \partial  W$}\}\in\mathbb{R},\\
K_0:=\min\{d\,:\, d\  \textrm{is the diameter of an open ball inscribed in}\ W\}>0.
\end{cases}
\end{equation*}
In our setting, it is clear that
\[
C_0 =\max\left\{C_1,\frac{1}{r}\right\}>0
\]
where
\[
C_1:=\max\{-\lambda\,:\,\lambda\ \textrm{is a curvature of $\partial \Omega$ at $x_0$ for $x_0 \in \partial  \Omega$}\}\in\mathbb{R}.
\]
In particular, if $\Omega$ is convex, then $C_0=1/r$.

We first have the following uniform Lipschitz estimate for $u$, which is strongly inspired by \cite{JKMT}.
\begin{theorem}\label{thm:Lip-u}
Assume that  there exists $\delta>0$ such that
\begin{equation}\label{condition:c}
c(x)^2-2|Dc(x)|-2C_0|c(x)|-\frac{8C_0}{K_0} \geq \delta \qquad\text{ for all } x\in W.
\end{equation}
Let $u$ be the unique viscosity solution to \eqref{eq:C}.
 Then, there exists a constant $L>0$ depending only on  $\|u_0\|_{C^2(\ol W)}$, $\|D\theta\|_{C^1(\ol W)}$, $\|c\|_{C^1(\ol W)}$,  the constants $C_0$, $K_0$, and $\delta$ such that
\begin{equation}\label{eq:lip}
\|u_t\|_{L^\infty(W\times [0,\infty))} + \|Du\|_{L^\infty(W\times [0,\infty))} \leq L.
\end{equation}
\end{theorem}

\begin{remark}
In case that $c(x)=c>0$ is a constant force term, then \eqref{condition:c} is equivalent to the requirement that
\[
c > C_0  \left(1 + \left(1+\frac{8}{C_0 K_0} \right)^{1/2} \right).
\]
Note that $C_0 \geq 1/r$, which is quite large for $r>0$ small.
Besides, $K_0$ can be quite small if some of the balls $\{B(a_j,r)\}$ are close to each other or close to $\partial \Omega$.
In such situations, we would need to require $c$ to be very large in order to have uniform Lipschitz estimate for $u$, which implies that \eqref{condition:c} is rather restrictive.
\end{remark}

\begin{example}
We consider the case that $\Omega = B(0,R)$, $W= B(0,R) \setminus B(0,r)$ for given $R\geq 2r>0$ and $c(x)=c>0$ is a constant force term.
Then, $C_0=1/r$, $K_0=R-r \geq r$, and \eqref{condition:c} is equivalent to the requirement that
\[
c > 4 C_0 = \frac{4}{r}.
\]
\end{example}

Next, we study the large time average of the auxiliary function $u$.
\begin{theorem}\label{thm:large average}
Let $u$ be the unique viscosity solution to \eqref{eq:C}.
For each $t\geq 0$, denote by
\[
S(t)=\max_{x\in \ol W} u(x,t).
\]
Then, there exists a constant $S=S_c \in \R$ such that
\[
\lim_{t \to \infty} \frac{S(t)}{t} = S_c.
\]
If we assume furthermore \eqref{condition:c}, then
\begin{equation}\label{eq:large average}
\lim_{t \to \infty} \frac{u(x,t)}{t}=S_c \qquad \text{ uniformly for } x\in \ol W.
\end{equation}
\end{theorem}

We now use the formulation described in \cite{OTG1} to deduce back the asymptotic growth rate of the crystal from the auxiliary function $u$.
Let $k(x,t) \in \Z$ be such that
\begin{equation}\label{eq:height1}
-\pi \leq u(x,t) - \left(\Theta(x) + 2\pi k(x,t) \right) <\pi,
\end{equation}
where $\Theta(x)=\sum_{j=1}^N m_j \Theta_j(x)$, and $\Theta_j: \ol W \to [0,2\pi)$ is the principal value of $\arg(x-a_j)$.
Denote by $h_0>0$ the unit height of steps.
We define the height of the crystal at $(x,t) \in \ol W \times [0,\infty)$ by
\begin{equation}\label{eq:height2}
h(x,t) = \frac{h_0}{2\pi} \left[ \Theta(x) + 2\pi k(x,t) + \pi\, {\rm sign} \left(u(x,t) - \left(\Theta(x) + 2\pi k(x,t) \right)\right) \right].
\end{equation}

\begin{theorem}\label{thm:crystal average}
Let $u$ be the unique viscosity solution to \eqref{eq:C}.
Let $S_c$ be defined as in Theorem \ref{thm:large average}.
Let $h:\ol W \times [0,\infty) \to \R$ be the height of the crystal defined by \eqref{eq:height1}--\eqref{eq:height2}.
Then,
\begin{equation}\label{eq:height3}
\lim_{t\to \infty} \frac{\max_{x\in \ol W} h(x,t)}{t} = \frac{h_0}{2\pi} S_c.
\end{equation}
If we assume furthermore \eqref{condition:c}, then
\begin{equation}\label{eq:height4}
\lim_{t \to \infty} \frac{h(x,t)}{t}=\frac{h_0}{2\pi} S_c \qquad \text{ uniformly for } x\in \ol W.
\end{equation}
\end{theorem}

In general, without assuming  \eqref{condition:c}, we do not have uniform control on the asymptotic growth rate of the crystal.
Nevertheless, having the asymptotic growth rate of the tip (the highest point) of the crystal in \eqref{eq:height3} might be enough for practical purposes.

\medskip

\begin{definition}
Let $S_c$ be as in Theorem \ref{thm:large average}.
We say that $S_c$ is the asymptotic growth rate of $u$ solving \eqref{eq:C}.
We also say that $h_0 S_c/(2\pi)$ is the asymptotic growth rate of the corresponding crystal whose height is defined by \eqref{eq:height1}--\eqref{eq:height2}.
\end{definition}

We next study the asymptotic growth rate $S_c$ of $u$, or equivalently, $h_0 S_c/(2\pi)$ of $h$.
Various interesting numerical results on the asymptotic growth rate were obtained in \cite{OTG1, OTG2}.

\begin{lemma}\label{lem:Sc-1}
Assume that $N=1$,  $m_1=1$, $a_1=0$, $W=B(0,R)\setminus B(0,r)$ for given $0<r<R$, and $c \in C^1(\ol W, (0,\infty))$.
Then,
\[
 \min_{x \in \ol W} \frac{c(x)}{|x|} \leq S_c \leq  \max_{x \in \ol W} \frac{c(x)}{|x|}.
\]
\end{lemma}

In Propositions \ref{lem:Sc-2}--\ref{lem:Sc-4} in Section \ref{sec:large}, we give various situations when $N=1,2$ in which $S_c=0$.

As the asymptotic growth rate of $u$ is $S_c$, it is natural to consider the following ergodic problem
\begin{equation}\label{eq:E}
\begin{cases}
-|D(v-\theta)| \left(\Div\left(\frac{D(v-\theta)}{|D(v-\theta)|}\right)+c(x)\right) = -S_c \qquad &\text{ in } W,\\
D(v-\theta)\cdot \mathbf{n}=0  \qquad &\text{ on } \partial W.
\end{cases}
\end{equation}

We have the following large time behavior result for \eqref{eq:C} in the case that $S_c=0$.

%\begin{lemma}\label{lem:E-1}
%If $S_c=0$, then \eqref{eq:E} does not admit classical solutions $v \in C^2(\ol W)$.
%\end{lemma}

\begin{theorem}\label{thm:C-E}
Let $u$ be the unique viscosity solution to \eqref{eq:C}.
Assume \eqref{condition:c} and $S_c=0$.
Then, there exists a viscosity solution $v$ to \eqref{eq:E} with $S_c=0$ such that
\[
\lim_{t\to \infty} \|u(\cdot,t)-v\|_{L^\infty(W)} =0.
\]
\end{theorem}

The proof of this theorem follows that of  \cite[Theorem 1.2]{GTZ} or \cite[Theorem 1.3]{JKMT} which uses a Lyapunov function, and is hence omitted.
It is worth emphasizing that the condition $S_c=0$ in Theorem \ref{thm:C-E} is essential for the argument; and if $S_c \neq 0$, then the large time behavior for \eqref{eq:C} is rather open.
Next, we give a result along the line of Lemma \ref{lem:Sc-1}.
For a given angle $\varphi$, denote by $\cR_\varphi$ a linear transformation given by rotating vectors through an angle of $\varphi$ counterclockwise in $\R^2$.
Then, the matrix of $\cR_\varphi$ is given by
\[
\cR_\varphi=
\begin{pmatrix}
\cos\varphi & -\sin\varphi \\
\sin\varphi & \cos\varphi
\end{pmatrix}.
\]

\begin{proposition}\label{prop:C-E}
Assume that $N=1$,  $m_1=1$, $a_1=0$, $W=B(0,R)\setminus B(0,r)$ for given $0<r<R$, and $c(x)=c_0|x|$ for $x\in \ol W$ for some $c_0>0$.
Assume further that $u_0(x)=g(x/|x|)$ for all $x\in \ol W$ for a given $g \in C^2(\R^2)$ with $\|Dg\|_{L^\infty} <1$.
Then, we have
\[
u(x,t) = g\left(\cR_{-c_0t} \frac{x}{|x|}\right) + c_0t \qquad \text{ for all } (x,t)\in \ol W \times [0,\infty).
\]
In particular, $u(x,t) - c_0t$ does not converge as $t\to\infty$ if $g$ is not constant.
\end{proposition}

Thus, in general, it is striking that solutions to \eqref{eq:C} do not converge to solutions of the ergodic problem \eqref{eq:E} after appropriate normalizations as time tends to infinity.

\subsection{Relevant literature}
We give a non exhaustive list of related works to our paper.
We refer the reader to \cite{CGG, ES, G} for the wellposedness of viscosity solutions of level-set mean curvature flows.
There have been many important papers using the level-set forced mean curvature flow PDE  to study spiral crystal growths \cite{Oh, GNOh, OTG1, OTG2, Oh2, Oh3, Sm}.
For Neumann boundary problems similar to this context, see \cite{H,GOS, Oh1,MT, MWW, JKMT, Jang}.
Besides the spiral crystal growth model, the birth and spread crystal growth model has also been studied extensively \cite{GMT2, GMT, GMOT, HM, G18}.
We emphasize that it is rather important to study the asymptotic growth rates and the shapes of the crystals in these models as time tends to infinity.
For related results concerning large time averages and large time behaviors, see \cite{GTZ, GLXY, MMTXY, FIM}.

\subsection*{Organization of the paper}
The paper is organized as follows. 
In Section \ref{sec:Lip}, we obtain the Lipschitz regularity of the solution of \eqref{eq:C} and give the proof of Theorem \ref{thm:Lip-u} under the additional assumption \eqref{condition:c}.
We then give the proof of Theorems \ref{thm:large average}--\ref{thm:crystal average} and study the properties of the asymptotic growth rate and the large time behavior of the solution in Section \ref{sec:large}.
In Section \ref{sec:ex}, we give an example of a non-uniformly Lipschitz continuous solution of \eqref{eq:C}, which shows that \eqref{condition:c} is needed in general if we want to have uniform estimates.

\subsection*{Notations}
For $1\leq j\leq N$, denote by $\theta_j(x)= \arg(x-a_j)$ and $\Theta_j$ the principle value of $\theta_j$.
For a given smooth function $\phi$, we write $\phi_k = \phi_{x_k}$, $\phi_{kl} = \phi_{x_k x_l}$.
When there is no confusion, we use the Einstein summation convention.
For $x=(x_1,x_2)\in \R^2$, we write $x^\perp=(-x_2,x_1)$.
For a given angle $\varphi$, denote by $\cR_\varphi$ a linear transformation given by rotating vectors through an angle of $\varphi$ counterclockwise in $\R^2$.
Note that $x^\perp=\cR_{\pi/2}\, x$ for $x\in \R^2$.

\subsection*{Acknowledgement}
The authors would like to extend their sincere appreciation to Professor Yoshikazu Giga for his insights and comments.
Additionally, the authors are grateful to Professor Takeshi Ohtsuka for sending them the references \cite{Oh2, Oh3} and his suggestions.

%%%%%%%%%%%%%%%%%%%%%%%%%%%%%%%%%%%%%%%%%%%%%%%%%%%%%%%%%%%%%%%%%%%%%

%%%%%%%%%%%%%%%%%%%%%%%%%%%%%%%%%%%%%%%%%%%%%%%%%%%%%%%%%%%%%%%%%%%%%

\section{Lipschitz regularity} \label{sec:Lip}

To obtain Lipschitz estimates, we consider the following approximation, for $\varepsilon \in (0,1)$, $T>0$,
\begin{equation}\label{eq:C-ep}
\begin{cases}
u^{\varepsilon}_t=\sqrt{\varepsilon^2+|D (u^{\varepsilon}-\theta)|^2}\left(\Div\left(\frac{D (u^{\varepsilon}-\theta)}{\sqrt{\varepsilon^2+|D (u^{\varepsilon}-\theta)|^2}}\right)+c(x) \right) \quad &\text{ in } W \times(0,T],\\
D(u^\ep-\theta)\cdot \mathbf{n}=0 \quad &\text{ on } \partial W \times[0,T],\\
u^{\varepsilon}(x,0)=u_0(x) \quad &\text{ on } \overline{W}.
\end{cases}
\end{equation}

The following result on a priori estimates on the gradient of $u^\ep$ implies right away Theorem \ref{thm:Lip-u}.

\begin{theorem}[A priori estimates]\label{thm:local-grad-ep}
Assume that $\partial W$ is smooth and $c\in C^\infty(\ol W)$.
Assume \eqref{condition:c}.
For each $\ep \in (0,1)$  and  $T>0$, assume that $u^\ep \in C^\infty(\ol W \times (0,T]) \cap C^1(\overline{W}\times[0,T])$ is the unique solution of \eqref{eq:C-ep}.
Then, there exists a constant $L>0$ depending only on $\|u_0\|_{C^2(\ol W)}$, $\|D\theta\|_{C^1(\ol W)}$, $\|c\|_{C^1(\ol W)}$,  the constants $C_0$, $K_0$, and $\delta$ from \eqref{condition:c} such that
\begin{equation}\label{eq:lip-u-ep}
\|u^\ep_t\|_{L^\infty(W\times [0,\infty))} + \|Du^\ep\|_{L^\infty(W\times [0,\infty))} \leq L.
\end{equation}

\end{theorem}

The proof below follows the ideas in \cite{JKMT}, which uses the classical Bernstein method.
We also refer the reader to \cite{Jang, Oh3}.
\begin{proof}
Firstly, let
\begin{equation}\label{eq:M}
M = 4 \lVert D^2(u_0-\theta)\rVert_{L^\infty(\ol W)} + \lVert  (1+|D (u_0-\theta)|^2)^{1/2} c\rVert_{L^\infty(\ol W)}.
\end{equation}
We see that $u_0 \pm Mt$ are a supersolution and a subsolution to \eqref{eq:C-ep}, respectively.
By the comparison of principle, we yield that, for $s\geq 0$,
\begin{equation}\label{eq:bdd-t-1}
u_0 - Ms \leq u^\ep(\cdot,s) \leq u_0 + Ms.
\end{equation}
Moreover, for each $s\geq 0$, $(x,t)\mapsto u^\ep(x,t+s)$ is a solution to \eqref{eq:C-ep} with initial data $u^\ep(x,s)$.
By using the comparison principle once more and \eqref{eq:bdd-t-1}, we deduce that, for $t,s\geq 0$,
\[
\|u^\ep(\cdot,t+s) - u^\ep(\cdot,t)\|_{L^\infty(W)} \leq \|u^\ep(\cdot,s) - u^\ep(\cdot,0)\|_{L^\infty(W)} \leq Ms.
\]
Therefore, we obtain that
\begin{equation}\label{eq:lip-u-ep-t}
\|u^\ep_t\|_{L^\infty(W\times [0,\infty))} \leq M.
\end{equation}

\medskip

We next show the boundedness of $Du^\ep$.
Denote by
\[
w^{\varepsilon}=\sqrt{\varepsilon^2+|D(u^{\varepsilon}-\theta)|^2}.
\] 
We only need to show that
\begin{equation}\label{eq:bound-w-ep}
\max_{\ol W \times [0,T]} w^\ep \leq C
\end{equation}
for some positive constant $C$ depending only on $\|u_0\|_{C^2(\ol W)}$, $\|D\theta\|_{C^1(\ol W)}$, $\|c\|_{C^1(\ol W)}$,  the constants $C_0$, $K_0$, and $\delta$ from \eqref{condition:c}.
The crucial point here is $C$ does not depend on $T>0$ and $\ep \in (0,1)$.
Pick $(x_0,t_0)\in\argmax_{\overline{W}\times[0,T]}w^{\varepsilon}$. 
If $t_0=0$, then
\[
\max_{\ol W \times [0,T]} w^\ep \leq w^\ep(x_0,0) \leq \|D(u_0-\theta)\|_{L^\infty(\ol W)} + 1,
\]
and \eqref{eq:bound-w-ep} holds true.

\smallskip

We consider the case $t_0>0$.
Write $u=u^{\varepsilon}$, $w=w^{\varepsilon}$ from now on for clarity.
Denote by
\[
b^{ij}=\delta_{ij}  - \frac{(u-\theta)_i (u-\theta)_j}{\varepsilon^2+|D(u-\theta)|^2}.
\]
Differentiate \eqref{eq:C-ep} with respect to $x_k$ and multiply the result by $(u-\theta)_k$ to get
\begin{multline*}
(u-\theta)_k(u-\theta)_{kt}-(D_pb^{ij}\cdot D(u-\theta)_k)(u-\theta)_k(u-\theta)_{ij}-b^{ij}(u-\theta)_{k}(u-\theta)_{kij}\\
-(u-\theta)_k c_k w-c\frac{(u-\theta)_k(u-\theta)_{lk}(u-\theta)_l}{w}=0.
\end{multline*}
In the above, we use the fact that $\theta_{kt}=0$ as $\theta$ is independent of $t$.
Substituting $ww_t=(u-\theta)_k(u-\theta)_{kt}$, $ww_k=(u-\theta)_l(u-\theta)_{kl}$, and $ww_{ij}=(u-\theta)_{kij}(u-\theta)_k+b^{kl}(u-\theta)_{ki}(u-\theta)_{lj}$, we yield
\begin{multline}\label{eq:w}
ww_t-w(D_pb^{ij}\cdot Dw)(u-\theta)_{ij}-wb^{ij}w_{ij}+b^{ij}b^{kl}(u-\theta)_{ki}(u-\theta)_{lj}\\
-wD(u-\theta)\cdot Dc-cD(u-\theta)\cdot Dw=0. 
\end{multline}
There are two cases to be considered $x_0\in W$ and $x_0\in\partial W$.

\medskip

\noindent {\bf Case 1: $x_0\in W$.}
We follow the computations of \cite[Lemma 4.1]{GMOT}. 
At $(x_0,t_0)$, we have $w_t\geq0$, $Dw=0$, $D^2w\leq0$, and thus
\[
wD(u-\theta)\cdot Dc\geq b^{ij}b^{kl}(u-\theta)_{ki}(u-\theta)_{lj}.
\]
We then use the Cauchy-Schwarz inequality
\[
(\textrm{tr}\alpha\beta)^2\leq\textrm{tr}(\alpha^2)\textrm{tr}(\beta^2)
\]
for all $\alpha, \beta\in\mathcal{S}^2$, and put $\alpha=A^{\frac{1}{2}}BA^{\frac{1}{2}}$, $\beta=I_2$, 
where $A=(b^{ij})$, $B=((u-\theta)_{kl})$, $I_2$ the identity matrix of size $2$ to get $\textrm{tr}(AB)^2\geq(\textrm{tr}AB)^2/\textrm{tr}(I_2).$

Therefore, at $(x_0,t_0)$,
\begin{align*}
\left|Dc(x_0)\right|w^2 &\geq wD(u-\theta)\cdot Dc\geq b^{ij}b^{kl}(u-\theta)_{ki}(u-\theta)_{lj}=\textrm{tr}(AB)^2\\
&\geq\frac{(\textrm{tr}AB)^2}{\textrm{tr}(I_2)}=\frac{1}{2}\left(u_t-c(x_0)w\right)^2.
\end{align*}
Since $\frac{1}{2}c(x)^2-\left|Dc(x)\right|\geq\frac{\delta}{2}>0$ by \eqref{condition:c}, we imply that at $(x_0,t_0)$,
\[
\delta w^2 \leq 2 u_t c(x_0) w \quad \Longrightarrow \quad w \leq \frac{2 M \|c\|_{L^\infty(\ol W)}}{\delta},
\]
which gives \eqref{eq:bound-w-ep}.

\medskip

\noindent {\bf Case 2: $x_0\in\partial W$.}
As $\partial W$ is $C^{2}$, we assume that $\mathbf{n}$ is defined as a $C^1$ function in a neighborhood of $\partial W$.
Note that the Neumann boundary condition $D(u-\theta) \cdot {\mathbf{n}}=0$ gives $\left(D^2(u-\theta){\mathbf{n}}+D{\mathbf{n}} D(u-\theta)\right)\cdot  v =0$ for all $ v \in \R^2$ perpendicular to ${\mathbf{n}}$  on $\partial W \times [0,T]$.
Thus, on $\partial W \times [0,T]$,
\[
\frac{\partial w}{\partial \mathbf{n}}=\frac{D^2(u-\theta) D(u-\theta)}{w}\cdot {\mathbf{n}}=-\frac{D{\mathbf{n}}D(u-\theta)\cdot D(u-\theta)}{w}\leq C_0\frac{|D(u-\theta)|^2}{w}.
\]
We note that $C_0 \geq 1/r$ and at $(x_0,t_0)$,
\[
\frac{\partial w}{\partial {\mathbf{n}}}\leq C_0\frac{|D(u-\theta)|^2}{w}<C_0w.
\]

\medskip

Pick $x_c\in W$ such that $B:=B(x_c, K_0/2) \subset W$ is tangent to $\partial W$ at $x_0$. 
Consider a multiplier for $w$
\[
\rho(x)=-\frac{C_0}{K_0}|x-x_c|^2+\frac{C_0 K_0}{4}+1 \quad \text{ for } x \in \ol W.
\]
Then, $\rho>1$ in $B$, $\rho=1$ on $\partial B$, and $\rho <1$ on $\ol W \setminus \ol B$.
Besides, $C_0 \rho(x_0)+\frac{\partial \rho}{\partial {\mathbf{n}}}(x_0)=0$. 

Denote by $\psi=\rho w$.
Then, at $(x_0,t_0)$,
\begin{equation}\label{rho-1}
\frac{\partial \psi}{\partial {\mathbf{n}}}=\frac{\partial (\rho w)}{\partial {\mathbf{n}}}=\rho\frac{\partial w}{\partial {\mathbf{n}}}+w\frac{\partial \rho}{\partial {\mathbf{n}}}<w\left(C_0\rho+\frac{\partial \rho}{\partial {\mathbf{n}}}\right) = 0.
\end{equation}
As noted above, for $(z,t) \in \left( \ol W \setminus \ol B \right)\times [0, T]$,
\begin{equation*}\label{rho-2}
 \psi(z,t) \leq w(z,t)\leq w(x_0,t_0) =  \psi(x_0,t_0),
\end{equation*}
and, by \eqref{rho-1},
\begin{equation}\label{rho-3}
\max_{\overline{W}\times[0,T]}\rho w = \max_{\overline{B}\times[0,T]}\rho w > \psi(x_0,t_0)=w(x_0,t_0).
\end{equation}
Let $(x_1,t_1)\in\argmax_{\overline{W}\times[0,T]}\rho w$.
Thanks to \eqref{rho-1}--\eqref{rho-3},  $x_1\in B \subset W$. 
 If $t_1=0$, then for all $(x,t)\in\overline{W}\times[0, T]$,
\begin{align*}
w(x,t)&\leq w(x_0,t_0)=\rho(x_0)w(x_0,t_0)\leq\rho(x_1)w(x_1,0) \\
&\leq  \left(\frac{C_0 K_0}{4}+1 \right)\left(\|D(u_0-\theta)\|_{L^\infty(\ol W)} + 1\right),
\end{align*}
which gives the desired result. 
We now consider the case that $t_1>0$. 
At this point $(x_1,t_1)$, we have $\psi_t\geq0,$ $D\psi=0,$ $D^2\psi\leq0$. 
Consequently, as $\psi_t=\rho w_t$, $D\psi=wD\rho+\rho Dw$, and $\psi_{ij}=w_{ij}\rho+w_i\rho_j+w_j\rho_i+w\rho_{ij}$, we have at $(x_1,t_1)$,
\[
w_t\geq-\frac{\rho_t}{\rho}w=0,\quad
 Dw=-\frac{w}{\rho}D\rho,\quad
w_{ij}=\frac{1}{\rho}(\psi_{ij}-w_i\rho_j-w_j\rho_i-w\rho_{ij}).
\]
We use \eqref{eq:w} and the above to yield that, at $(x_1,t_1)$,
\begin{multline*}
\frac{w^2}{\rho}(D_{p}b^{ij}\cdot D\rho)(u-\theta)_{ij}+\frac{w}{\rho}b^{ij}(w_i\rho_j+w_j\rho_i+w\rho_{ij})\\
+b^{ij}b^{kl}(u-\theta)_{ki}(u-\theta)_{lj}-wD(u-\theta)\cdot Dc+\frac{cw}{\rho}D(u-\theta)\cdot D\rho\leq0.
\end{multline*}
By direct computations
\[
b^{ij}_{p_l}=-\frac{\delta_{il}(u-\theta)_j}{\varepsilon^2+|D(u-\theta)|^2}-\frac{\delta_{jl}(u-\theta)_i}{\varepsilon^2+|D(u-\theta)|^2}+\frac{2(u-\theta)_i(u-\theta)_j(u-\theta)_l}{(\varepsilon^2+|D(u-\theta)|^2)^2},
\]
and hence,
\begin{align*}
&w(D_pb^{ij}\cdot D\rho)(u-\theta)_{ij}\\
=\, &w\left(-\frac{\rho_i (u-\theta)_j(u-\theta)_{ij}}{\varepsilon^2+|D(u-\theta)|^2}-\frac{\rho_j(u-\theta)_i(u-\theta)_{ij}}{\varepsilon^2+|D(u-\theta)|^2}+\frac{2(u-\theta)_i(u-\theta)_j(u-\theta)_l\rho_l (u-\theta)_{ij}}{(\varepsilon^2+|D(u-\theta)|^2)^2}\right)\\
=\, &-2Dw\cdot D\rho+\frac{2(D(u-\theta)\cdot D\rho)(D(u-\theta)\cdot Dw)}{w^2}.
\end{align*}
Therefore,
\begin{align*}
&w(D_pb^{ij}\cdot D\rho)(u-\theta)_{ij}+b^{ij}w_i\rho_j+b^{ij}w_j\rho_i\\
=\ &\frac{2(D(u-\theta)\cdot D\rho)(D(u-\theta)\cdot Dw)}{w^2}-\frac{(u-\theta)_i(u-\theta)_jw_i\rho_j}{w^2}-\frac{(u-\theta)_i(u-\theta)_jw_j\rho_i}{w^2}\\
=\ &0.
\end{align*}
Thus, at $(x_1,t_1)$, 
\begin{equation}\label{rho-4}
\frac{\rho_{ij}}{\rho}b^{ij}w^2+b^{ij}b^{kl}(u-\theta)_{ki}(u-\theta)_{lj}-wD(u-\theta)\cdot Dc+\frac{cw}{\rho}D(u-\theta)\cdot D\rho\leq 0.
\end{equation}
Using the Cauchy-Schwarz type inequality as in the above, we deduce that
\begin{align*}
\frac{1}{2}(u_t-c(x_1)w)^2&\leq b^{ij}b^{kl}(u-\theta)_{ki}(u-\theta)_{lj}\\
&\leq-\frac{w^2}{\rho}b^{ij}\rho_{ij}+wD(u-\theta)\cdot Dc-\frac{cw}{\rho}D(u-\theta)\cdot D\rho\\
&\leq \frac{2C_0}{K_0}\frac{w^2}{\rho} \left(2-\frac{|D(u-\theta)|^2}{\varepsilon^2+|D(u-\theta)|^2}\right)+|Dc|w^2+C_0|c|w^2\\
&\leq\left(\frac{4C_0}{K_0}+|Dc(x_1)|+C_0|c(x_1)|\right)w^2.
\end{align*}
Combine the above with \eqref{condition:c} that
\[
\frac{1}{2}c(x)^2-|Dc(x)|-C_0 |c(x)|-\frac{4 C_0}{K_0}\geq\frac{\delta}{2}>0\quad\quad\text{for all}\ x\in\overline{W},
\]
we see that $w(x_1,t_1) \leq \frac{2 M \|c\|_{L^\infty(\ol W)}}{\delta}$. 
Thus,
\[
w(x_0,t_0) \leq \rho(x_1) w(x_1,t_1) \leq \left(\frac{C_0 K_0}{4}+1\right) \frac{2 M \|c\|_{L^\infty(\ol W)}}{\delta}.
\]

\end{proof}

%%%%%%%%%%%%%%%%%%%%%%%%%%%%%%%%%%%%%%%%%%%%%%%%%%%%%%%%%%%%%%%%%%%%%

\section{Large time averages and behaviors} \label{sec:large}
\subsection{Large time average results}
We first prove Theorem \ref{thm:large average}.
Our approach is similar to that in \cite{GMOT}.

\begin{proof}[Proof of  Theorem \ref{thm:large average}]
We note that $u \in C(\ol W \times [0,\infty))$ and in light of \eqref{eq:bdd-t-1},
\[
u_0(x) - Mt \leq u(x,t) \leq u_0(x) +Mt \qquad \text{ for all } (x,t) \in \ol W \times [0,\infty),
\]
where $M$ is the constant defined in \eqref{eq:M}.
Thus, $S \in C([0,\infty))$, and $|S(t)| \leq \|u_0\|_{L^\infty(\ol W)} + Mt$ for $t\geq 0$.

Fix $s \geq 0$. 
By the definition of $S$, $u(x,s) \leq u_0(x) +\|u_0\|_{L^\infty(\ol W)} + S(s)$ for $x\in \ol W$.
Then, by the comparison principle for \eqref{eq:C}, we imply that, for $x\in \ol W$ and $t\geq 0$,
\[
u(x,s+t) \leq u(x,t) +\|u_0\|_{L^\infty(\ol W)} + S(s).
\]
In particular,
\[
S(s+t) \leq S(t) + S(s) +\|u_0\|_{L^\infty(\ol W)}.
\]
Thus, $t\mapsto  S(t) + \|u_0\|_{L^\infty(\ol W)}$ is subadditive.
By the Fekete lemma, there exists $S=S_c \in \R$ such that
\[
\lim_{t\to\infty} \frac{ S(t) + \|u_0\|_{L^\infty(\ol W)}}{t} = \lim_{t\to\infty} \frac{ S(t) }{t} = S.
\]
In fact,
\[
S = \inf_{t>0} \frac{S(t) + \|u_0\|_{L^\infty(\ol W)}}{t},
\]
which yields that $|S| \leq M$.

Let us assume further that \eqref{condition:c} holds.
Then, by Theorem \ref{thm:Lip-u}, $\|Du\|_{L^\infty(\ol W \times [0,\infty))} \leq L$.
In particular, for $(x,t) \in \ol W \times [0,\infty)$,
\[
|u(x,t) - S(t)| \leq C.
\]
Hence, uniformly for $x\in \ol W$,
\[
\lim_{t\to\infty} \frac{u(x,t)}{t} = \lim_{t\to\infty} \frac{ S(t) }{t} = S,
\]
which gives \eqref{eq:large average}.
\end{proof}

\begin{proof}[Proof of  Theorem \ref{thm:crystal average}]
Notice first that $\|\Theta\|_{L^\infty(\ol W)} \leq 2\pi \sum_{j=1}^N |m_j|$.
Thus, for $(x,t)\in \ol W \times (0,\infty)$,
\[
-\frac{C}{t} \leq \frac{u(x,t)}{t} - \frac{2 \pi k(x,t)}{t} \leq \frac{C}{t},
\]
and
\[
-\frac{C}{t} \leq \frac{h(x,t)}{t} - \frac{h_0 k(x,t)}{t} \leq \frac{C}{t},
\]
where $C= (2\pi + h_0) \left(\sum_{j=1}^N |m_j|+1\right)$.
\end{proof}

\subsection{Properties of the asymptotic growth rate}
We next provide further analysis on the properties of the asymptotic growth rate $S_c$ of $u$.

\begin{proof}[Proof of Lemma \ref{lem:Sc-1}]
It is enough to obtain the lower bound of $S_c$ as the proof for the upper bound of $S_c$ follows analogously. 
Set $\gam= \min_{x\in \ol W} c(x)/|x|>0$.
We claim that 
\[
\phi(x,t) = -\|u_0\|_{L^\infty(\ol W)} + \gam t
\]
 is a viscosity subsolution to \eqref{eq:C}, which will automatically give us that $S_c \geq \gamma$ by the comparison principle.

Let us now prove this claim.
It is obvious that $\phi(\cdot,0) \leq u_0$.
Next, note that $D(\phi-\theta)(x,t) = -D\theta(x) = -{x^\perp}/{|x|^2}$, and hence,
\[
\Div\left(\frac{D(\phi-\theta)}{|D(\phi-\theta)|}\right)=-\Div\left(\frac{x^\perp}{|x|}\right)=0.
\]
Thus,
\begin{align*}
\phi_t -|D(\phi-\theta)| \left(\Div\left(\frac{D(\phi-\theta)}{|D(\phi-\theta)|}\right)+c(x)\right)
=\gam - \frac{c(x)}{|x|} \leq 0.
\end{align*}
The proof is complete.
\end{proof}

\begin{remark}
Let us note that in the situation of Lemma \ref{lem:Sc-1}, it is also possible to construct subsolutions and supersolutions to \eqref{eq:C} of the form
\[
\phi(x,t) = \varphi\left(\frac{x}{|x|}\right) + \gamma t,
\]
where $\varphi$ is a smooth function to be decided.
Nevertheless, the lower and upper bounds of $S_c$ are unchanged, regardless of the specific choice of subsolutions and supersolutions in this form.
\end{remark}

\begin{remark}
In the situation of Lemma \ref{lem:Sc-1}, if we assume further that $c(x)=c_0|x|$ for all $x\in \ol W$ for some given constant $c_0>0$, then $\phi(x,t)=\alpha + c_0 t$ is a viscosity solution to \eqref{eq:C}  with $\phi(x,0)=\alpha$ for any $\alpha \in \R$.
Moreover, in this particular case, we have that $S_c=c_0$, which is independent of $r$, the radius of the deleted hole.
\end{remark}

Next, we consider the situation where there are two spiral centers which rotate in different directions of the same strength.
This was first studied in \cite{BCF} in which the two spirals are called an opposite rotating pair.
If the two spiral centers are close to each other, that is, $|a_1-a_2| \leq 2/\max_{\ol W} c$, then they are called an inactive pair, in which case  $S_c=0$. (see Proposition \ref{lem:Sc-3}).
If the two spiral centers are far away from each other, \cite{BCF} gives some heuristic explanations that this opposite rotating pair actually accelerates the growth rate, which is faster than that of a single spiral.
There has not been a rigorous proof of this point in the literature yet.
We refer the reader to \cite{OTG1, OTG2} for numerical results.

We give several results along this line.
The first one is a toy model case (see Figure \ref{fig.two-holes-strip}).

\begin{proposition}\label{lem:Sc-2} 
Assume that $N=2$, $a_1=(l,0)$, $a_2=(-l,0)$ with $0<r<l$, and
\[
\big[B(a_1,r) \cup B(a_2,r) \cup ((-l,l)\times (-r,r))\big] \cap W=\emptyset.
\]
Assume furthermore that $m_1=-m_2$. 
Then, $S_c=0$.
\end{proposition}

 \begin{center}
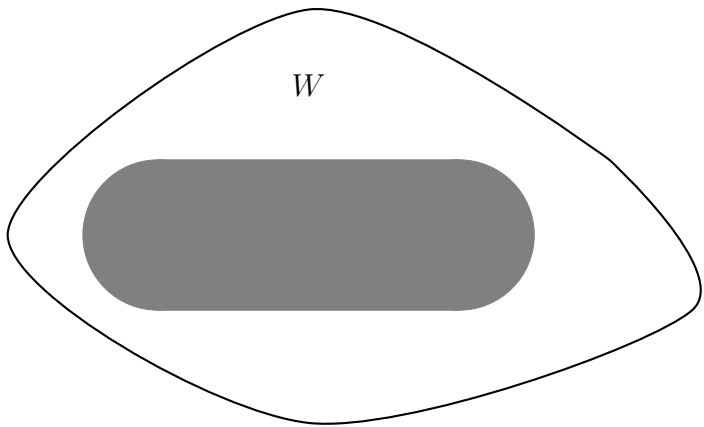
 
\begin{tikzpicture}
\draw[thick] plot [smooth] coordinates{(4,1) (0,3) (-4,0) (0,-2.5) (5.1,-1) (4,1)};
\draw[gray,  fill=gray](2,0) circle (1);
\draw[gray, fill=gray](-2,0) circle (1);
\draw[gray, fill=gray] (-2,1) rectangle (2,-1);
\fill (0,2) node {$W$};
\end{tikzpicture}
\captionof{figure}{An example of $W$ in Proposition \ref{lem:Sc-2}.}
\label{fig.two-holes-strip}
\end{center}

We now consider the situation where the two spiral centers are close to each other (see Figure \ref{fig.two-holes-only}).
\begin{proposition}\label{lem:Sc-3}
Assume that $N=2$, $a_1=(l,0)$, $a_2=(-l,0)$ with 
\[
0<r<l \leq R_0:=\frac{1}{\max_{\ol W} c},
\]
and $m_1=-m_2$. 
Pick $\theta \in (0,\pi/2]$ such that $R_0 = (l-r \cos \theta)/\sin\theta$.
Denote by $A=(-l,0)+ r(\cos \theta, \sin \theta)$, 
$B=(l,0)+r(-\cos \theta, \sin \theta)$,
$C=(-l,0)+r(\cos \theta, - \sin \theta)$,
$D=(l,0)+r(-\cos \theta, - \sin \theta)$.
Let $\overset{\frown}{AB}$ be an arc of a circle of radius $R_0$ that is perpendicular to both $\partial B(a_1,r)$ and $\partial B(a_2,r)$.
Let $\overset{\frown}{CD}$ be an arc of a circle of radius $R_0$ that is perpendicular to both $\partial B(a_1,r)$ and $\partial B(a_2,r)$.
Assume furthermore that $\overset{\frown}{AB}, \overset{\frown}{CD} \subset W$.
Then, $S_c=0$.
\end{proposition}
When $c(\cdot)$ is constant, Proposition \ref{lem:Sc-3} was already announced in \cite{Oh2}.

 \begin{center}
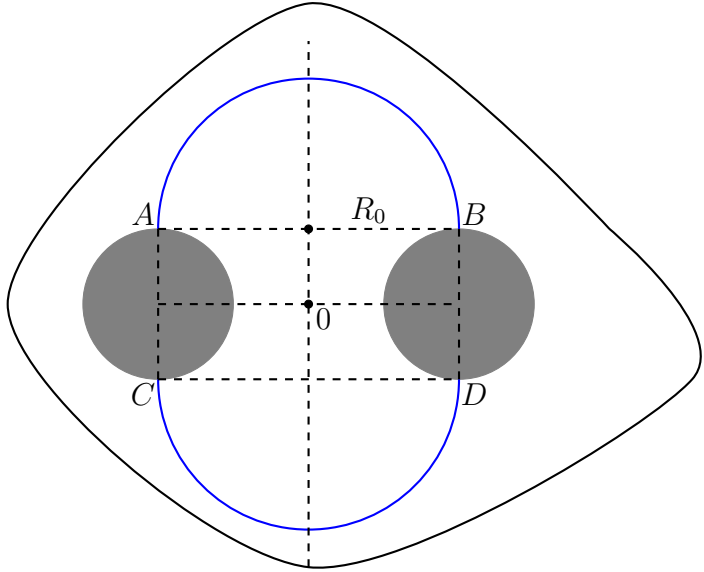
 
\begin{tikzpicture}
\draw[thick] plot [smooth] coordinates{(4,1) (0,4) (-4,0) (0,-3.5) (5.1,-1) (4,1)};
\draw[gray,  fill=gray](2,0) circle (1);
\draw[gray, fill=gray](-2,0) circle (1);
\draw[thick, blue] (2,1) arc(0:180:2);
\draw[thick, blue] (-2,-1) arc(180:360:2);
\fill (-2.2,1.2) node {$A$} (2.2,1.2) node {$B$} (-2.2,-1.2) node {$C$} (2.2,-1.2) node {$D$} (0.2,-0.2) node {$0$} (0.8,1.25) node {$R_0$};
\draw[thick, dashed] (-2,0)--(2,0);
\draw[thick, dashed] (-2,-1)--(-2,1);
\draw[thick, dashed] (2,-1)--(2,1);
\draw[thick, dashed] (-2,1)--(2,1);
\draw[thick, dashed] (-2,-1)--(2,-1);
\draw[thick, dashed] (0,-3.5)--(0,3.5);
\draw [fill] (0,0) circle [radius=1.5pt];
\draw [fill] (0,1) circle [radius=1.5pt];
\end{tikzpicture}
\captionof{figure}{An example of $W$ in Proposition \ref{lem:Sc-3} with $\theta=\pi/2$.}
\label{fig.two-holes-only}
\end{center}

By using a similar argument as that in the proof of Proposition \ref{lem:Sc-3}, we obtain the following rather surprising result with $N=1$ (see Figure \ref{fig.one-hole-only}).
\begin{proposition}\label{lem:Sc-4}
Assume that $N=1$, $m_1=1$.
Let $R_0=1/{\max_{\ol W} c}$.
Assume that there exist two arcs $\overset{\frown}{AB}, \overset{\frown}{CD}$ of two circles of radii $R_0$ with $A,B \in \partial B(a_1,r)$,  $\overset{\frown}{AB}, \overset{\frown}{CD}\subset W$ such that they are both perpendicular to $\partial B(a_1,r)$ and $\partial \Omega$.
Assume furthermore that $\overset{\frown}{AB}, \overset{\frown}{CD}$, $\partial B(a_1,r)$, and $\partial \Omega$ separate $W$ into two connected components $W_1, W_2$ such that $\overset{\frown}{AB}, \overset{\frown}{CD}$ are convex on $\partial W_1$, and $\arg(x-a_1)$ is well defined in each of $W_1$ and $W_2$.
Then, $S_c=0$.
\end{proposition}
We note that Proposition \ref{lem:Sc-4} is related to Theorem \ref{thm:no-Lip}, in which we give an example of a non-uniformly Lipschitz continuous solution of \eqref{eq:C}.

\begin{remark}
It is not hard to construct other examples with $N \geq 2$ with general $m_1,\ldots, m_N\in \Z \setminus \{0\}$ that are similar to those in Propositions \ref{lem:Sc-3}--\ref{lem:Sc-4} with $S_c=0$.
We thus see that the asymptotic growth rate $S_c$ demonstrates notable sensitivity to the geometric characteristics of the boundary $\partial \Omega$.
\end{remark}

 \begin{center}
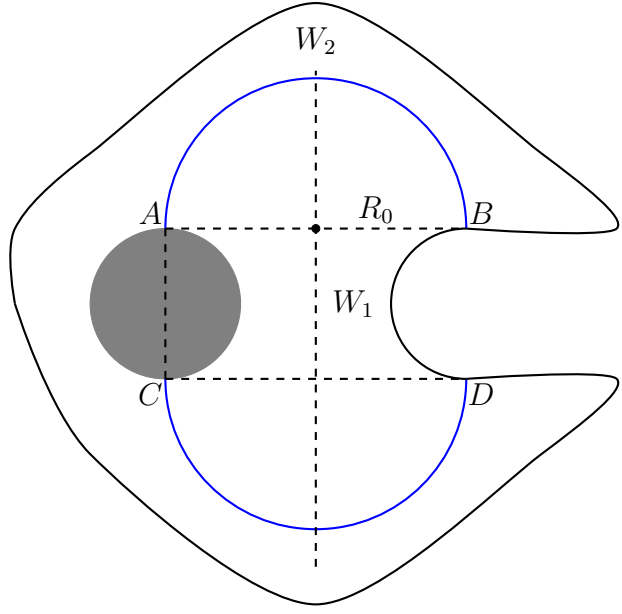
 
\begin{tikzpicture}
\draw[gray, fill=gray](-2,0) circle (1);
\draw[thick, blue] (2,1) arc(0:180:2);
\draw[thick, blue] (-2,-1) arc(180:360:2);
\fill (-2.2,1.2) node {$A$} (2.2,1.2) node {$B$} (-2.2,-1.2) node {$C$} (2.2,-1.2) node {$D$} (0.5,0) node {$W_1$} (0,3.5) node {$W_2$} (0.8,1.25) node {$R_0$};
\draw[thick] (2,1) arc(90:270:1);
\draw[thick] plot [smooth] coordinates{(-4,0) (-3,-2) (0,-4) (3,-2) (4,-1) (2,-1)};
\draw[thick] plot [smooth] coordinates{(-4,0) (-4,1) (-3,2) (0,4) (3,2) (4,1) (2,1)};
\draw[thick, dashed] (-2,-1)--(2,-1);
\draw[thick, dashed] (-2,-1)--(-2,1);
\draw[thick, dashed] (-2,1)--(2,1);
\draw[thick, dashed] (0,-3.5)--(0,3.1);
\draw [fill] (0,1) circle [radius=1.5pt];
\end{tikzpicture}
\captionof{figure}{An example of $W$ in Proposition \ref{lem:Sc-4}.}
\label{fig.one-hole-only}
\end{center}

\begin{proof}[Proof of Proposition \ref{lem:Sc-2}]
It is very interesting that in this case 
\[
\theta(x):= 
\sum_{j=1}^2 m_j\arg(x-a_j)
\]
is well-defined in $W$ as $m_1=-m_2$ and $a_1=(l,0), a_2 = (-l,0)$.
It is thus easy to see that 
\[
\varphi^\pm (x,t) =\pm (\|u_0\|_{L^\infty(\ol W)}+\|\theta\|_{L^\infty(\ol W)}) + \theta(x) 
\]
is a supersolution and a subsolution to \eqref{eq:C}, respectively.
Therefore, for $(x,t)\in \ol W \times [0,\infty)$,
\[
-(\|u_0\|_{L^\infty(\ol W)}+\|\theta\|_{L^\infty(\ol W)})+ \theta(x) \leq u(x,t) \leq (\|u_0\|_{L^\infty(\ol W)}+\|\theta\|_{L^\infty(\ol W)}) + \theta(x) .
\]
We then get the desired result.
\end{proof}

We now give a proof of Proposition \ref{lem:Sc-3}.
\begin{proof}[Proof of Proposition \ref{lem:Sc-3}]
Without loss of generality, assume $m_1=-m_2=1$.
 Let $W_1$ be the region inside $W$ enclosed by the arcs $\overset{\frown}{AB}, \overset{\frown}{CD}$ and $\partial B(a_1,r), \partial B(a_2,r)$.
 Let $W_2 = W \setminus \overline{W_1}$.
 Note that
 \[
\theta(x):= \arg(x-a_1) - \arg(x-a_2)
\]
 is well-defined in $W_1$ and $W_2$ separately.
 Let $v=\theta+2\pi k$ be such that, with an appropriate choice of branch cuts, we have for $w=v-\theta$, then
 \[
 w(x)=
 \begin{cases}
 2\pi \qquad &\text{ for } x\in W_1,\\
 0 \qquad &\text{ for } x\in W_2.
 \end{cases}
 \]
 See Figure \ref{fig.two-holes-only-proof}.
 
 \begin{center}
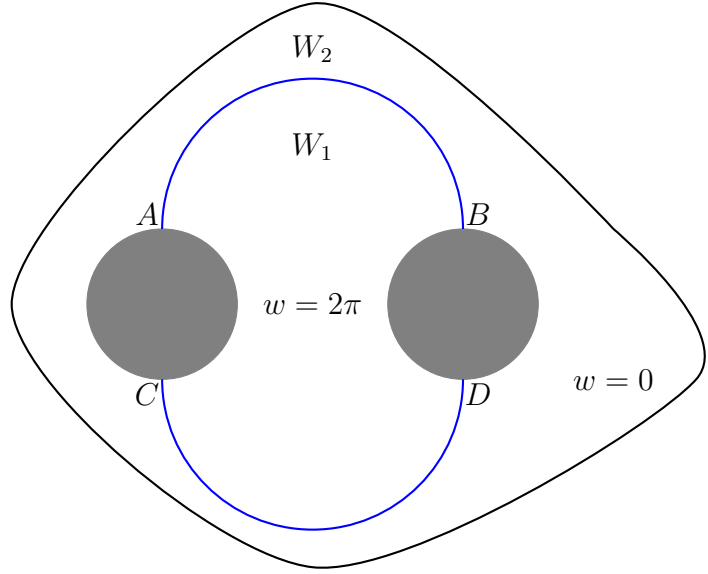
 
\begin{tikzpicture}
\draw[thick] plot [smooth] coordinates{(4,1) (0,4) (-4,0) (0,-3.5) (5.1,-1) (4,1)};
\draw[gray,  fill=gray](2,0) circle (1);
\draw[gray, fill=gray](-2,0) circle (1);
\draw[thick, blue] (2,1) arc(0:180:2);
\draw[thick, blue] (-2,-1) arc(180:360:2);
\fill (-2.2,1.2) node {$A$} (2.2,1.2) node {$B$} (-2.2,-1.2) node {$C$} (2.2,-1.2) node {$D$} (0,2.1) node{$W_1$} (0,3.4) node{$W_2$};
\fill (0,0) node {$w=2\pi$} (4,-1) node {$w=0$};
\end{tikzpicture}
\captionof{figure}{Values of $w$ in $W$ in Proposition \ref{lem:Sc-3}.}
\label{fig.two-holes-only-proof}
\end{center}

Surely, $v$ is discontinuous across the arcs $\overset{\frown}{AB}, \overset{\frown}{CD}$.
We now show that $v$ is a supersolution to \eqref{eq:C} by proving that $w=v-\theta$ is a supersolution to
\[
\begin{cases}
-|Dw| \left(\div\left(\frac{Dw}{|Dw|}\right) + \frac{1}{R_0} \right) = 0 \qquad &\text{ in } W,\\
Dw\cdot \mathbf{n} = 0 \qquad &\text{ on } \partial W.
\end{cases}
\]
We only need to check the supersolution property on the open arcs $\overset{\frown}{AB}, \overset{\frown}{CD}$ and at the boundary points $A,B,C,D$.

First, let $\phi$ be a smooth test function such that $\phi$ touches $w_*$ from below at a point $x_0 \in \overset{\frown}{AB} \cap W$.
Thanks to \cite[Lemma A.1]{GMT2}, we get that there exists $s \leq 0$ such that
\[
|D\phi(x_0)| = |s| \quad \text{ and } \quad |D\phi(x_0)| \div\left(\frac{D\phi(x_0)}{|D\phi(x_0)|}\right) \leq \frac{s}{R_0},
\]
which gives the desired supersolution test.

Second, let $\phi$ be a smooth test function such that $\phi$ touches $w_*$ from below at $A$.
Parameterize the arc $\overset{\frown}{AB}$ by a smooth curve $\xi$ such that $\xi(0)=A$ and $\xi(1)=B$.
Then we have $\phi(\xi(0))=0$ and $\phi(\xi(t)) \leq 0$ for all $t\in [0,1]$, which imply
\[
D\phi(\xi(0))\cdot \dot \xi(0) \leq 0.
\]
Therefore, at $A$,
\[
D\phi(A)\cdot \mathbf{n} \geq 0.
\]
Thus, $v$ is a supersolution to \eqref{eq:C}.
By the usual comparison principle, we yield that, for $(x,t)\in \ol W \times [0,\infty)$,
\[
u(x,t) \leq \|u_0\|_{L^\infty(\ol W)} + \|v\|_{L^\infty(\ol W)} + v(x).
\]
The proof is complete.
\end{proof}

We give a proof of Proposition \ref{lem:Sc-4}, which is similar to that of Proposition \ref{lem:Sc-3}.
\begin{proof}[Proof of Proposition \ref{lem:Sc-4}]
By our assumptions, we have that
 \[
\theta(x):= \arg(x-a_1)
\]
 is well-defined in $W_1$ and $W_2$ separately.
 Let $v=\theta+2\pi k$ be such that, with an appropriate choice of branch cuts, we have for $w=v-\theta$, then
 \[
 w(x)=
 \begin{cases}
 2\pi \qquad &\text{ for } x\in W_1,\\
 0 \qquad &\text{ for } x\in W_2.
 \end{cases}
 \]
 See Figure \ref{fig.one-hole-only-proof}.

 \begin{center}
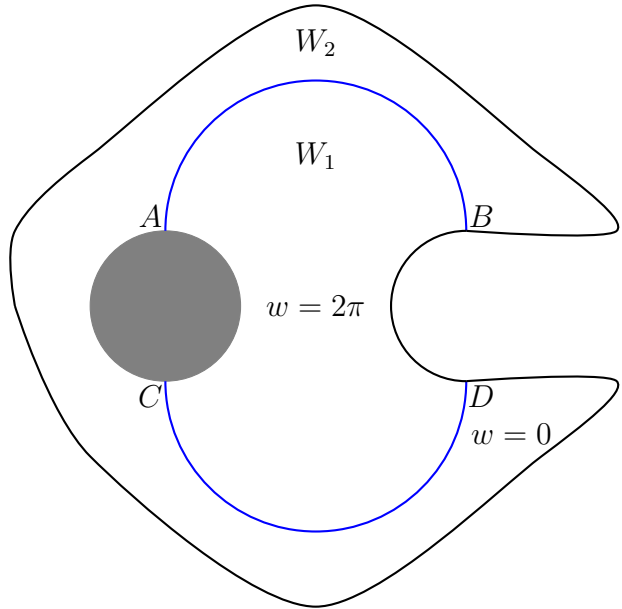

\begin{tikzpicture}
\draw[gray, fill=gray](-2,0) circle (1);
\draw[thick, blue] (2,1) arc(0:180:2);
\draw[thick, blue] (-2,-1) arc(180:360:2);
\fill (-2.2,1.2) node {$A$} (2.2,1.2) node {$B$} (-2.2,-1.2) node {$C$} (2.2,-1.2) node {$D$};
\draw[thick] (2,1) arc(90:270:1);
\draw[thick] plot [smooth] coordinates{(-4,0) (-3,-2) (0,-4) (3,-2) (4,-1) (2,-1)};
\draw[thick] plot [smooth] coordinates{(-4,0) (-4,1) (-3,2) (0,4) (3,2) (4,1) (2,1)};
\fill (0,2) node {$W_1$} (0,3.5) node{$W_2$} (0,0) node {$w=2\pi$} (2.6,-1.7) node {$w=0$};
\end{tikzpicture}
\captionof{figure}{Values of $w$ in $W$ in Proposition \ref{lem:Sc-4}.}
 \label{fig.one-hole-only-proof}
\end{center}

Of course, $v$ is discontinuous across the arcs $\overset{\frown}{AB}, \overset{\frown}{CD}$.
We now show that $v$ is a supersolution to \eqref{eq:C} by proving that $w=v-\theta$ is a supersolution to
\[
\begin{cases}
-|Dw| \left(\div\left(\frac{Dw}{|Dw|}\right) + \frac{1}{R_0} \right) = 0 \qquad &\text{ in } W,\\
Dw\cdot \mathbf{n} = 0 \qquad &\text{ on } \partial W.
\end{cases}
\]
We only need to check the supersolution property on the open arcs $\overset{\frown}{AB}, \overset{\frown}{CD}$ and at the boundary points $A,B,C,D$.

First, let $\phi$ be a smooth test function such that $\phi$ touches $w_*$ from below at a point $x_0 \in \overset{\frown}{AB} \cap W$.
Thanks to \cite[Lemma A.1]{GMT2}, we get that there exists $s \leq 0$ such that
\[
|D\phi(x_0)| = |s| \quad \text{ and } \quad |D\phi(x_0)| \div\left(\frac{D\phi(x_0)}{|D\phi(x_0)|}\right) \leq \frac{s}{R_0},
\]
which gives the desired supersolution test.

Second, we need to check the supersolution property at the boundary points $A,B,C,D$. 
It is enough to check it at $B$ as the proof is analogous for other points.
let $\phi$ be a smooth test function such that $\phi$ touches $w_*$ from below at $B$.
Parameterize the arc $\overset{\frown}{AB}$ by a smooth curve $\eta$ such that $\eta(0)=B$ and $\eta(1)=A$.
Then we have $\phi(\eta(0))=0$ and $\phi(\eta(t)) \leq 0$ for all $t\in [0,1]$, which imply
\[
D\phi(\eta(0))\cdot \dot \eta(0) \leq 0.
\]
Therefore, at $B$,
\[
D\phi(B)\cdot \mathbf{n} \geq 0.
\]
Thus, $v$ is a supersolution to \eqref{eq:C}.
By the usual comparison principle, we yield that, for $(x,t)\in \ol W \times [0,\infty)$,
\[
u(x,t) \leq \|u_0\|_{L^\infty(\ol W)} + \|v\|_{L^\infty(\ol W)} + v(x).
\]
The proof is complete.
\end{proof}

\subsection{Large time behaviors}
The proof of Theorem \ref{thm:C-E} follows that of  \cite[Theorem 1.2]{GTZ} or \cite[Theorem 1.3]{JKMT} by using a standard Lyapunov function, and is hence omitted.
We remark that assuming both \eqref{condition:c} and $S_c=0$ is really restrictive.

\smallskip

Let us now prove Proposition \ref{prop:C-E}.

\begin{proof}[Proof of Proposition \ref{prop:C-E}]
This can be done by direct computations. 
Set
\[
u(x,t) = g\left(\cR_{-c_0t} \frac{x}{|x|}\right) + c_0t \qquad \text{ for all } (x,t)\in \ol W \times [0,\infty).
\]
Then, for $(x,t)\in \ol W \times [0,\infty)$,
\begin{align*}
u_t &= \frac{d}{dt}(\cR_{-c_0t}) Dg \cdot \frac{x}{|x|} + c_0\\
&=c_0 \left( 1 - Dg \cdot \left (\cR_{\pi/2 -c_0t} \frac{x}{|x|} \right) \right)
=c_0 \left( 1 - Dg \cdot \left (\cR_{-c_0t} \frac{x^\perp}{|x|} \right) \right),
\end{align*}
and
\[
Du=\frac{\cR_{-c_0t}}{|x|} \left(Dg\cdot \frac{x^\perp}{|x|} \right) \frac{x^\perp}{|x|}.
\]
Therefore, for $(x,t)\in \ol W \times [0,\infty)$, we use the assumption that $\|Dg\|_{L^\infty}<1$ to yield
\[
|D(u-\theta)|= \frac{1}{|x|} \left( 1 - Dg \cdot \left (\cR_{-c_0t} \frac{x^\perp}{|x|} \right) \right).
\]
As $c(x)=c_0|x|$ and $\div(x^\perp/|x|)=0$, we imply the conclusion.
\end{proof}

%%%%%%%%%%%%%%%%%%%%%%%%%%%%%%%%%%%%%%%%%%%%%%%%%%%%%%%%%%%%%%%%%%%%%

\section{An example of a non-uniformly Lipschitz solution to \eqref{eq:C}} \label{sec:ex}
In this section we construct a non-uniformly Lipschitz continuous solution to \eqref{eq:C}.
Let us consider $c\equiv 1$, in which case \eqref{eq:C} becomes
\begin{equation}\label{eq:example}
\begin{cases}
u_t =|D(u-\theta)| \left(\Div\left(\frac{D(u-\theta)}{|D(u-\theta)|}\right)+1\right) \qquad &\text{ in } W\times(0,\infty), \\
D(u-\theta)\cdot \mathbf{n}=0  \qquad &\text{ on } \partial W\times(0,\infty),\\
u(\cdot,0)=u_0 \qquad &\text{ on }  \ol W,
\end{cases}
\end{equation}
for a domain $W\subset\R^2$ to be constructed. 
Fix $r>0$, and take $R>r$, where $R$ will be chosen later.  
Let $\bm{a}:=(-r,0)\in\partial B(0,r)$, take $\overline{\bm{a}}:=(-r,\overline{a}_2)\in\partial B(0,R)$ with $\overline{a}_2>0$. 
Note that $(\overline{\bm{a}}-\bm{a})\cdot\bm{a}=0$. 
Let $\overline{\bm{b}}:=(-\frac{r}{\sqrt{2}}, \frac{r}{\sqrt{2}})$. Take $\overline{\bm{b}}=(\overline{b}_1, \overline{b}_2)\in\partial B(0,R)$ 
so that 
\[
(\overline{\bm{b}}-\bm{b})\cdot\bm{b}=0, \quad \overline{b}_2>0. 
\]
Set $\overline{\bm{c}}:=(\overline{\bm{a}}+\overline{\bm{b}})/{2}$. 
Take $\bm{c}=(c_1,c_2)\in\partial B(0,r)$ so that 
\[
(\overline{\bm{c}}-\bm{c})\cdot\bm{c}=0, \quad c_1\in \left(-r,-\frac{r}{\sqrt{2}}\right). 
\]
Note that $\bm{a}, \overline{\bm{a}}, \bm{b}, \overline{\bm{b}}, \bm{c}, \overline{\bm{c}}$ are uniquely determined.  

Set $l:=|\bm{c}-\overline{\bm{c}}|$, and note that $l=l(R)$ depends on $R>0$. 
By the construction, it is clear to see that $l(R_1)<l(R_2)$ if $0<R_1<R_2$. 
Since $l(R)\to0$ as $R\to r$, $l(R)\to\infty$ as $R\to\infty$, 
there exists $R>0$ such that $l(R)=1$. 
Henceforth, we fix such a $R>0$. 
We set 
\[
\left[\overline{\bm{c}}, \overline{\bm{b}}\right]
:=
\left\{
\overline{\bm{d}}(s):=\overline{\bm{c}}+s\left(\overline{\bm{b}}-\overline{\bm{c}}\right)\,:\, s\in[0,1]
\right\}. 
\]
Take $\bm{d}(s)=(d_1(s),d_2(s))\in\partial B(0,r)$ so that 
\[
(\overline{\bm{d}}(s)-\bm{d}(s))\cdot\bm{d}(s)=0, \quad d_1(s)\in \left(c_1,-\frac{r}{\sqrt{2}}\right). 
\]
We define a family of arcs by 
\[
\Gamma(s):=
\left\{
\overline{\bm{d}}(s)+
\cR_\varphi
(\bm{d}(s)-\overline{\bm{d}}(s))
\,:\, \varphi\in\left[0,\frac{\pi}{2}\right]
\right\}. 
\]
Recall that
\[
\cR_\varphi=
\begin{pmatrix}
\cos\varphi & -\sin\varphi \\
\sin\varphi & \cos\varphi
\end{pmatrix}.
\]
It is important to note that $|\bm{d}(s)-\overline{\bm{d}}(s)|>1$ for $s \in (0,1]$.
Set the vector field by 
\[
\bm{n}(x):=\frac{x-\overline{\bm{d}}(s)}{|x-\overline{\bm{d}}(s)|}
\qquad\text{ for }  x\in\Gamma(s). 
\]
Note that, by construction, it is clear that 
$\bm{n}$ is smooth. Therefore, for a fixed $z\in\Gamma(0)$, there exists $t_0>0$ such that 
the initial value problem of the ordinary differential equation 
\[
\left\{
\begin{array}{ll} 
\dot{\bm{\gamma}}(t)=\bm{n}(\bm{\gamma}(t)) \qquad & \text{ for} \ t\in(0,t_0), \\
\bm{\gamma}(0)=z
\end{array}
\right. 
\]
has a solution at least for a short time. 
Write $\bm{\gamma}(t)=(\gamma_1(t), \gamma_2(t))$.
For  $t_0>0$ is small enough, $\gamma_2(t)>0$ for $t\in(0, t_0)$. 
Take $s_0\in(0,1)$ so that $\bm{\gamma}(t_0)\in\Gamma(s_0)$. 
Next, we extend $\bm{\gamma}=(\gamma_1, \gamma_2)$ to be a $C^2$ curve satisfying 
there exist $t_1<0<t_0<t_2$ such that 
\[
\left\{
\begin{array}{ll} 
\gamma_2(t)>0 \qquad \text{ for all} \ t\in(t_1,t_2), \\
\gamma_2(t_1)=\gamma_2(t_2)=0, \\
\dot \gamma_2(t_1)=\dot \gamma_2( t_2)=0. 
\end{array}
\right. 
\]
Finally, we define $\partial\Omega\subset\R^2$ by 
\begin{equation}\label{eq:Omega}
\partial\Omega:=
\left\{
(\gamma_1(t), \gamma_2(t))\,:\, t\in[t_1, t_2]
\right\}
\cup 
\left\{
(\gamma_1(t), -\gamma_2(t))\,:\, t\in[t_1, t_2]
\right\}, 
\end{equation}
and we denote by $W$ the domain enclosed by $\partial\Omega\cup\partial B(0,r)$. 
Here, we assume $N=1$, $a_1=0$, $m_1=1$.
Then, $\theta(x)=\arg(x)$.

 \begin{center} 
\includegraphics[width=4in]{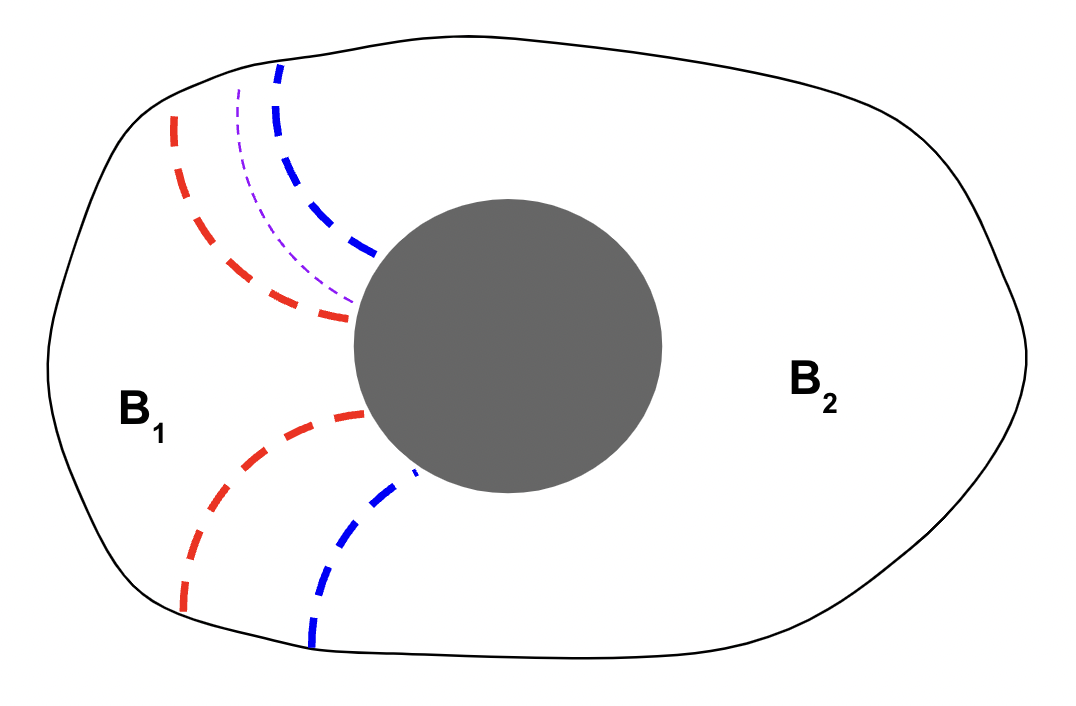}
\captionof{figure}{The constructed domain $W$.}
\label{fig.no-Lip}
\end{center}

Now, let $0<\alpha<\beta<2\pi$, and we define $u_0\in\Lip(\overline{W})$ by 
\begin{equation}\label{eq:u0}
u_0(x_1, x_2)-\theta(x):=
\left\{
\begin{array}{ll} 
\alpha+\frac{s}{s_0}(\beta-\alpha) \quad &\text{for} \ (x_1, x_2)\in\Gamma(s), \ s\in[0,s_0], \\
\alpha+\frac{s}{s_0}(\beta-2\pi-\alpha) \quad &\text{for} \ (x_1, x_2)\in\tilde{\Gamma}(s), \ s\in[0,s_0], 
\end{array}
\right. 
\end{equation}
where we set $\tilde{\Gamma}(s):=\{(x_1, -x_2)\,:\, (x_1, x_2)\in\Gamma(s)\}$.

We set 
\[
\cA:=\bigcup_{0\le s\le s_0}\Gamma(s), \qquad 
\tilde \cA:=\bigcup_{0\le s\le s_0}\tilde\Gamma(s).
\]
Take $x_1>r$ such that $\pm( x_1,0)\in\Omega$. 
We denote by $B_1$ and $B_2$ the connected component in $\ol W\setminus (\cA \cup\tilde \cA)$ which include $(-x_1,0)$ and $(x_1,0)$, respectively. 
We set 
\begin{equation}\label{eq:u01}
u_0(x_1, x_2)-\theta(x):=
\left\{
\begin{array}{ll} 
\alpha \qquad & \text{ for} \ (x_1, x_2)\in B_1 \\
\beta \qquad & \text{ for} \ (x_1, x_2)\in B_2.  
\end{array}
\right. 
\end{equation}
We note that $\Gamma(0)$ and $\tilde{\Gamma}(0)$ are exactly the two arcs $\overset{\frown}{AB}, \overset{\frown}{CD}$ of two circles of radii $1$ required in Proposition \ref{lem:Sc-4}.
For $s\in (0,s_0]$, $\Gamma(s)$ is an arc of a circle of radius greater than $1$.
Naturally, under the normal velocity $V=\kappa+1$, all the arcs $\Gamma(s)$ for $s\in (0,s_0]$ converge to $\Gamma(0)$ as time tends to infinity.
See Figure \ref{fig.no-Lip}.
By the above construction, the result in the following theorem follows immediately.
\begin{thm}\label{thm:no-Lip}
Assume $N=1$, $a_1=0$, $m_1=1$.
Let $\Omega$ be the set defined by \eqref{eq:Omega}, and let $u_0\in\Lip(\overline{W})$ 
be the function defined by \eqref{eq:u0} and \eqref{eq:u01}. 
Let $u$ be the solution to \eqref{eq:example}. Then, 
\[
\lim_{t\to\infty} (u(x,t) - \theta(x))
=
\left\{
\begin{array}{ll} 
\alpha \qquad & \text{ for} \ (x_1, x_2)\in B_1 \\
\beta \qquad & \text{ for} \ (x_1, x_2)\in W \setminus \overline{B_1}.  
\end{array}
\right. 
\]
In particular, $S_c=0$.
\end{thm}

We skip the proof of this theorem and we refer the reader to \cite[Section 6]{JKMT} for detail on a rather similar result.

%%%%%%%%%%%%%%%%%%%%%%%%%%%%%%%%%%%%%%%%%%%%%%%%%%%%%%%%%%%%%%%%%%%%%

\section*{Data availability}
Data sharing not applicable to this article as no datasets were generated or analyzed during the current study.

\section*{Conflict of interest}
There is no conflict of interest.

\end{document}